\documentclass[12pt]{article}
\usepackage{amsmath,amsfonts,amssymb,amsthm,amscd}
\title{The Picard-Vessiot theory, constrained cohomology, and linear differential algebraic groups}
\date{\today}
\author{Anand Pillay\thanks{Partially supported by NSF grant DMS-1360702}\\University of Notre Dame}
\newtheorem{Theorem}{Theorem}[section]

\newtheorem{Definition}[Theorem]{Definition} 
\newtheorem{Remark}[Theorem]{Remark}
\newtheorem{Lemma}[Theorem]{Lemma}
\newtheorem{Corollary}[Theorem]{Corollary}
\newtheorem{Fact}[Theorem]{Fact}

\begin{document}
\maketitle

\begin{abstract} We prove that a differential field $K$ is algebraically closed and Picard-Vessiot closed if and only if the differential Galois cohomology group, $H_{\partial}^{1}(K,G)$, is trivial for any linear differential algebraic group $G$  over $K$.  We give  applications to the parameterized Picard-Vessiot theory and pose several problems. 
\end{abstract}

\section{Introduction and preliminaries}
In this paper we are mainly concerned with differential fields $(K,\partial)$ of characteristic $0$ and we try to relate two ``theories", the  Galois theory of linear differential equations over $K$ (i.e. the Picard-Vessiot theory), and Kolchin's ``constrained cohomology" (which we prefer to  call ``differential Galois cohomology") of  linear differential algebraic groups over $K$.  We work in the category of differential algebraic varieties in the sense of Kolchin \cite{Kolchin-DAG}. Because we are concerned  with differential algebraic groups and their torsors, we may equally well (and in fact do) work in the category of definable sets in the first order theory $DCF_{0}$ of differentially closed fields (with respect to a single derivation $\partial$) of characteristic $0$. 
 The main result is the following, where we explain the notions later. 
\begin{Theorem} Let $(K, \partial)$ be a differential field. Then $K$ is algebraically closed and Picard-Vessiot closed if and only if $H^{1}_{\partial}(K,G) = \{1\}$ for any linear differential algebraic group $G$ over $K$, namely if every differential algebraic principal  homogeneous space $X$ for $G$, defined over $K$, has a $K$-point.
\end{Theorem}

\begin{Remark}
(i) A Picard-Vessiot (or $PV$) extension $L$ of $K$ is by definition a differential field extension $L$ of $K$ generated over $K$ by a fundamental system of solutions $Y_{1},..,Y_{n}$ of a linear differential system $\partial Y = AY$ over $K$, such that $C_{L} = C_{K}$, namely $L$ has ``no new constants".  So formally, to be $PV$-closed means that $K$ has no proper $PV$ extension $L$. 
\newline
(ii) When $C_{K}$ is algebraically closed (in particular when $K$ is algebraically closed), then there always exists a $PV$ extension of $K$ for any linear differential sytem $\partial Y = AY$, hence to be Picard-Vessiot closed is equivalent to $K$ containing a fundamental system of solutions for any linear system $\partial  Y = AY$ over $K$.
\newline
(iii) When $C_{K}$ is algebraically closed then $K$ being $PV$-closed implies $K$ algebraically closed.
\newline
(iv)  For $G$ a differential algebraic group over $K$, $H^{1}_{\partial}(K,G)$ is defined in \cite{Kolchin-DAG}, Chapter VII,  in terms of certain cocycles from ${\mathcal G} = Aut)K^{diff}/K)$ to $G(K^{diff})$, where $K^{diff}$ is the differential closure of $K$ inside some ``universal" differentially closed field    $\mathcal U$. 
$H^{1}_{\partial}(K, G)$ parametrizes the set of differential algebraic principal homogeneous spaces $X$ for $G$ over $K$ up to isomorphism over $K$ (in the appropriate category). 
So triviality of $H^{1}_{\partial}(K,G)$ means that for every such $PHS$ $X$ for $G$ over $K$,  $X(K)\neq \emptyset$. 
\end{Remark} 

We will also point out how our proof of Theorem 1.1 extends to prove the following:
\begin{Theorem} Let $K$ be a differential field.  Then $K$ is algebraically closed and closed under {\em generalized strongly normal extensions}  if and only if $H^{1}_{\partial}(K,G) = \{1\}$ for any (not necessarily linear) differential algebraic group $G$ over $K$. 
\end{Theorem}

Theorem 1.3 can be seen as evidence for the naturalness of the generalized strongly normal theory. 

\vspace{2mm}
\noindent
Theorem 1.1 has the following consequence for the {\em parameterized Picard-Vessiot theory} \cite{Cassidy-Singer-PPV}, where again  the additional  notions will be explained subsequently. 
\begin{Corollary} Let $K$ be a $\{\partial_{x}, \partial_{t}\}$-differential field where $\partial _{x}$ and $\partial{t}$ are commuting derivations. Let $K^{\partial_{x}}$ denote the field of constants of $K$ with respect to $\partial_{x}$. Let $\partial_{x}Y = AY$ be a parametrized linear differential equation over $K$. Assume that $K^{\partial_{x}}$ is both algebraically closed and $PV$-closed as a $\partial_{t}$-field. Then $K$ has a unique parameterized Picard Vessiot extension for the equation. 
\end{Corollary}

\vspace{5mm}
\noindent
Although our results are stated in purely differential algebraic terms, we will freely use a mixture of differential algebraic and model-theoretic notations and methods  in the proofs.  In fact, differential algebraic geometry in the sense of Kolchin is synonymous with the model theory of differentially closed fields, in terms of content and subject matter, although model theory brings with it a certain generality in point of view. 

This paper is concerned mainly with ordinary differential fields of characteristic $0$, namely fields $K$ of characteristic $0$ equipped with a single derivation, although in Corollary 1.4 and section 3 we touch on the case of a set $\Delta$ of commuting derivations.  For the majority of the introduction we will focus on $\Delta = \{\partial\}$.  For differential algebraic geometry, namely the theory (and category) of differential algebraic varieties defined  over a differential field $K$, the reader could consult \cite{Kolchin-DAG}, although some of the axiomatic presentation is a bit obscure.  But Section 1.1 of \cite{Cassidy-Singer-Jordan-Holder}  can also serve as an introduction to this material.  The reader is also referred to Buium's book \cite{Buium} for an approach somewhat more informed by modern algebraic geometry. Key objects of this paper are differential algebraic groups $G$ and their principal homogeneous spaces (or torsors) $X$. Namely $G$ is a differential algebraic group defined over $K$, $X$ is a differential algebraic variety defined over $K$, asnd we have a regular (strictly transitive) group  action of $G$ on $X$ (on the left) which is a morphism $G\times X \to X$ over $K$ defined over $K$ in the sense of differential algebraic varieties. 

$DCF_{0}$ is the first order theory of differentially closed fields of characteristic $0$ in the language$L = \{ +,-,0,1, \times, \partial\}$. It is complete, $\omega$-stable,  with quantifier elimination, and elimination of imaginaries.  We take ${\cal U}$ to be a ``saturated" model of $DCF_{0}$ in which all differential fields we are concerned with are assumed to embed. ${\cal U}$ is precisely a ``universal" differential field in the sense of Kolchin.  $K$ will usually denote a ``small" differential subfield of $\cal U$.   ${\cal C}$ denotes the field of constants of ${\cal U}$, and $C_{K}$ the field of constants of $K$. But when we work with partial fields (several commuting derivations) notation will be a bit different. The reader is referred to \cite{Marker} for the basic model theory of $DCF_{0}$.   By a $K$-definable set we mean a subset of ${\cal U}^{n}$ which is definable with parameters from $K$  in the $L$-structure $\cal U$. Our paper \cite{Pillay-Foundations} gives an exhaustive account of Kolchin's theory of differential algebraic varieties and groups, and its relation to the category of definable sets and definable groups, and the interested reader is referred there. 
We summarize: 
\begin{Fact}
(i)  any group $G$ definable over $K$ has a unique structure of a differential algebraic group over $K$.
\newline
 (ii) If $G$ is a definable group over $K$ and $X$ is a principal homogeneous space for $G$ over $K$, in the category of definable sets over $K$, then $X$ can be given (uniquely) the structure of a differential algebraic variety over $K$ such that $X$ is a $PHS$ for $G$ over $K$ in the category of differentiak algebraic varieties over $K$,
\newline
(iii)  If $(G,X)$ is a definable (differential algebraic) $PHS$ over $K$ then there is an embedding (over $K$) in an algebraic $PHS$ $(G_{1},X_{1})$ over $K$. 
\end{Fact}

Also connectedness of groups has the same meaning in the definable and differential algebraic categories.

\vspace{5mm}
\noindent
An important notion in this paper is the {\em differential closure} $K^{diff}$ of a differential field $K$.  It is the analogue in differential algebra of the algebraic closure $K^{alg}$ of a field $K$, and from the point of view of model theory is the ``prime model over $K$".   Now $Aut(K^{alg}/K)$ is a profinite group, but $Aut(K^{diff}/K)$ does not have such a nice ``geometric"structure. 
As in the case of usual Galois cohomology, the set of $PHS$'s over $K$ for a given differential algebraic group $G$ over $K$ can be described in terms of suitable cocycles from $Aut(K^{diff}/K)$ to $G(K^{diff})$, modulo the equivalence relation of being cohologous.  This is put into a somewhat more general model-theoretic environment in \cite{Pillay-Galoiscohomology}, where it is pointed out that {\em definable} cocycles  suffice. For the purposes of the current paper, these descriptions are not needed.  In any case $H^{1}_{\partial}(K, G)$ denotes the set of (differential algebraic) $PHS$'s for $G$, over$K$ up to isomorphism over $K$, where the isomorphism should respect the $G$-action.  We will make use of Kolchin's theorem (Theorem 4, Chapter VII, \cite{Kolchin-DAG})  which says:
\begin{Fact} Suppose $G$ is an algebraic group over $K$. Then $H^{1}_{\partial}(K,G) = H^{1}(K, G)$. In particular if $K$ is algebraically closed then $H^{1}(K,G) = \{1\}$. 
\end{Fact}

\begin{Definition}(i) A linear differential algebraic group ($LDAG$)over $K$ is a subgroup of $GL_{n}({\cal U})$ which is differential algebraic over $K$, equivalently definable over $K$ in ${\cal U}$. 
\newline
(ii) An Abelian differential algebraic group over $K$ is a differential algebraic (equivalently definable) over $K$ subgroup of $A({\cal U})$ for some Abelian variety $A$ over $K$. 
\end{Definition} 

Due to Fact 1.5 (iii)  and the Chevalley decomposition for algebraic groups we have:
\begin{Fact} Let $G$ be a connected differential algebraic group over $K$. Then we have an exact sequence $1\to C \to G \to B \to 1$  of connected differential algebraic groups over $K$ where $C$ is linear and $B$ is Abelian. 
\end{Fact}

We will make use of the following result from \cite{Cassidy-LDAG}:
\begin{Fact}
 If $G$ is a $LDAG$ over $K$ and $N$ is a normal definable (over $K$) subgroup, then $G/N$ is an $LDAG$ over $K$. 
\end{Fact}

The analogous result for Abelian $DAG$'s fails. 

\vspace{5mm}
\noindent

Model theory provides various ordinal-valued dimensions attached to definable sets and types in $DCF_{0}$, the most important for our purposes being Morley rank $RM$ and $U$-rank.  In $DCF_{0}$ every complete type $p(x)$ (in finitely many variables) has $U$-rank $< \omega^{2}$.  The $U$-rank of a definable set is the supremum of the $U$-ranks of complete types extending that definable set.  For definable groups $G$ in $DCF_{0}$, $RM(G) = U(G)$ and moreover the supremum above is attained by the so-called ``generic types" of $G$.  (See \cite{Pillay-Pong}).   We will call $G$ finite-dimensional if $RM(G) < \omega$. This is  Buium's nomenclature \cite{Buium} and corresponds to $G$ having "differential type $0$" in the notation of \cite{Cassidy-Singer-Jordan-Holder} for example.   In general the Morley rank and $U$-rank of a definable group $G$ will have the form $\omega\cdot m + k$, for $m, k$ nonnegative integers.  The ``leading coefficient" (which is $k$ when $m=0$ and $m$ otherwise) is related to but distinct from  the ``typical differential dimension" in the Kolchin theory. 

We will be making use of a number of results about $U$-rank and definable groups from \cite{Poizat-stablegroups} (some of which come from the seminal paper \cite{Berline-Lascar})with full references.

The model theoretic environment for partial differential algebraic geometry is the theory $DCF_{0,m}$ of differentially closed fields with respect to $m$ commuting derivations $\partial_{1}, ..., \partial_{m}$.  Essentially everything we have said above about the ordinary case passes over to the partial case, except now the $U$-ranks (and Morley ranks) of types are bounded by $\omega^{m+1}$, namely every complete type has the form $\omega^{m}k_{m} +  \omega^{m-1}k_{m-1} + ... + \omega k_{1} + k_{0}$, where $k_{0},..,k_{m}$ are nonnegative integers.

Finally we discuss the Galois theoretic notions which are central to this paper.  We are mainly concerned with the Picard-Vessiot theory, but we mention ``generalized strongly normal extensions" in Theorem 1.3, and the parameterized Picard-Vessiot ($PPV$)  theory, with respect to  $DCF_{0,2}$  in Corollary 1.4.

We repeat part of Remark 1.2. A linear differential equation over $K$, in vector form, is something of the form 
$$(*)           \partial Y = AY$$ where $Y$ is a $n\times 1$ (column) vector and $A$ is an $n\times n$ matrix over $K$.  The set $V$ of solutions in $\cal U$  is a subset of ${\cal U}^{n}$ which is an $n$-dimensional $\cal C$-vector space. What we call a fundamental system of solutions is precisely a $\cal C$-basis of $V$, and is the same thing as a matrix $Z\in GL_{n}({\cal U})$ whose columns are solutions of (*).  A $PV$ extension of $K$ for (*) is a differential field extension $L$ of the form $K(Z)$ for some fundamental system $Z$ of solutions, such that $C_{K} = C_{L}$.  $L$ is always contained in some differential closure $K^{diff}$ of $K$, and when $C_{K}$ is algebraically closed, such a $PV$ extension $L$ exists and is unique. 

A first generalization of the notion of Picard-Vessiot extension is  that of a {\em strongly normal extension} and is due to Kolchin \cite{Kolchin-DA-AG}. $L$ is a strongly normal extension of $K$ if for some $\alpha$ such that $L$ is generated over $K$ by $\alpha$, we have (i) for any realization $\beta$ of $tp(\alpha/K)$, $\beta$ is in the differential field generated by $K$, $\alpha$ and ${\cal C}$, and (ii)  $C_{K} = C_{L}$. 

A further generalization is the notion of a ``generalized strongly normal extension of $K$" which was introduced in \cite{DGTI} and elaborated on in  \cite{DGTII}, \cite{DGTIII} and \cite{DGTIV}.  Let $X$ be a $K$-definable set (such as the field of constants). $L = K\langle \alpha\rangle$ is an $X$-strongly normal extension of $K$ if (i) for any realization $\beta$ of $tp(\alpha/K)$, $\beta$ is in the differential field generated by $\alpha$, $K$, and (the coordinates of elements of) $X$, and  (ii) $X(K) = X(K^{diff})$.  (We should mention that condition (ii) has been weakened in \cite{Leon-Sanchez-Pillay}, although this does not affect the current paper.) 

When $K$ is algebraically closed, generalized strongly normal extensions (as well as strongly normal extensions) can be explicitly described in terms of solving certain kinds of differential equations, as follows:  Given a connected algebraic group over $K$, a $\partial$-structure on $G$ is an extension of the derivation $\partial$ on $K$ to a derivation $D$ of the structure sheaf (coordinate ring when $G$ is affine) of $G$. Such a derivation $D$ is equivalent to a  homomorphic regular section $s$ over $K$ of a certain ``shifted tangent bundle" $T_{\partial}G$ of $G$.
 We call $(G,s)$ an algebraic $\partial$-group over $K$, which is an object belonging to algebraic geometry (and was introduced and studied in \cite{Buium}). It gives rise to a map (in fact crossed homomorphism) $dlog_{(G,s)}$ belonging to differential algebraic geometry, from $G$ to its Lie algebra $LG$, namely the map taking $y\in G(\cal U)$ to  $\partial(y)s(y)^{-1}$ where the latter is computed in the algebraic group $T_{\partial}(G)$. We let $(G, s)^{\partial}$ be the kernel of $dlog_{(G,s)}$, a finite-dimensional differential algebraic group over $K$, and if $a\in LG(K)$, we call the equation $dlog_{(G,s)}(y) = a$ a logarithmic differential equation over $K$ on $(G,s)$. 
 In any case, it is shown in \cite{DGTIV} that (assuming $K$ to be algebraically closed), $L$ is a generalized strongly normal extension of $K$ (i.e. an $X$-strongly normal extension of $K$ for some $X$) iff there is a connected algebraic $\partial$-group $(G,s)$ over $K$, a logarithmic  differential equation $dlog_{(G,s})(-) = a\in LG(K)$ over $K$ and a solution $y\in G(K^{diff})$ of the equation such that (i) $L = K(y)$ and (ii) $(G,s)^{\partial}(K) = (G,s)^{\partial}(K^{diff})$.    The case when $G$ is over $C_{K}$ and $s$ is the ``trivial" $\partial$-structure on $G$  recovers Kolchin's strongly normal theory.

\section{Proofs of Main results}
We use notation and conventions from the previous section.  We use repeatedly the fact that  $C_{K^{diff}} = C_{K}^{alg}$, in particular if $C_{K}$ is algebraically closed then it coincides with the constants of $K^{diff}$.

\vspace{5mm}
\noindent
{\bf Proof of Theorem 1.1.}
\newline
We first point out the (trivial) right to left direction. If $K$ is not algebraically closed, then a proper Galois extension of $K$ gives a finite principal homogeneous space $X$ over $K$ for a finite group $G$, such that $X(K) = \emptyset$. So we may assume $K$ is algebraically closed. If $K$ is not $PV$-closed $K$ has a proper $PV$extension for a system $\partial Y = AY$ over $K$. $L$ is generated by a fundamental system $b = (b_{1},..,b_{n})$ of solutions for the equation.
The map taking $\sigma\in Aut_{\partial}(L/K)$ to $\sigma(b)b^{-1}$  (multiplication in $GL_{n}$) does not depend on the choice of $b$, and establishes an isomorphism between $Aut_{\partial}(L/K)$ and a differential algebraic subgroup $G$ of $GL_{n}$ defined over $K$.   The orbit of $b$ under $G$ (where $G$ acts on the left) is a differential algebraic left $PHS$ for $G$, defined over $K$, and with no $K$-point, contradicting our assumptions.  (The above construction of $G$ should be considered well-known but details are also given in Section 3 of \cite{Kamensky-Pillay}. )

We now prove the left to right direction.  We are assuming that $K$ is algebraically closed and $PV$ closed. We let $G$ be an arbitrary linear differential algebraic group over $K$ and we aim  to prove that $H^{1}_{\partial}(K,G) = \{1\}$. 

We will be using the following:

\begin{Lemma} [{\bf INDUCTIVE PRINCIPLE}] 
Suppose that we have a short exact sequence $1 \to N \to G \to H \to 1$ of  differential algebraic groups over $K$. Suppose that $H^{1}_{\partial}(K,H) = \{1\}$ and $H^{1}_{\partial}(K,N) = \{1\}$. Then $H^{1}_{\partial}(K,G) = \{1\}$. 
\end{Lemma}

\begin{proof} This follows from the usual exact sequence in cohomology, but we can be more explicit: Suppose $X$ is a differential algebraic $PHS$ for $G$ over $K$.  Then the set $Y$ of $N$-orbits in $X$ is a differential algebraic $PHS$ for $H$ over $K$. If $Y$ has a $K$-point, then this yields an $N$-orbit $Z\subseteq X$, defined over $K$, and then a $K$-point of $Z$ yields a $K$-point of $X$. 
\end{proof}

\vspace{5mm}
\noindent
There are several ways to proceed, modulo Lemma 2.1 above, depending on what and how much to quote of the  differential algebraic and model-theoretic literature.  Cassidy and Singer \cite{Cassidy-Singer-Jordan-Holder}, working in the context of several commuting derivations,  prove that an arbitrary ``strongly connected" differential algebraic group over $K$ has a subnormal series (of differential algebraic groups over $K$) where the successive quotients are ``almost simple". Moreover in the ordinary case (one derivation) they describe the ``almost simple" linear differential algebraic groups, so it would be  enough to check Theorem  1.1 for each of these.   In fact we will take a slightly different route, with more model-theoretic input,  notation, and information about the structure of linear differential algebraic groups, but ending up checking Theorem 1.1 in essentially the same special cases (and more or less recovering the decomposition theorem of \cite{Cassidy-Singer-Jordan-Holder}). 

\vspace{5mm}
\noindent
{\bf Case (1)}  $G$ is ``finite-dimensional" (so  finite Morley rank, or finite $U$-rank).  By the inductive principle we may assume $G$ is connected. Let $N$ be the ``definable solvable radical of $G$", namely the maximal normal solvable connected  definable subgroup.  It exists and is unique as $G$ has finite Morley rank, in particular $N$ is defined over $K$. Then clearly $H = G/N$ is defined over $K$, linear (by Fact 1.9), and {\em semisimple} in the sense that $H$ has no connected normal abelian definable subgroup. This notion of semisimplicity agrees with the notion in \cite{Cassidy-semisimple}. 

By Lemma 2.1, we can separate into subcases:
\newline
{\bf Case (1)(a).}  $G = N$, namely $G$ is solvable. 

As $G$ is a differential algebraic subgroup of $GL_{n}({\cal U})$ defined over $K$, we can consider $\bar G$,  the Zariski closure of $G$, a connected solvable linear algebraic group over $K$ (in the algebraic geometric sense). By the results in Section III.10 of \cite{Borel}, $\bar 
G$ is filtered by a chain of normal algebraic subgroups (over $K$) each successive quotient being isomorphic over $K$ to the additive group ${\mathbb G}_{a}$ or multiplicative group ${\mathbb G}_{m}$.  Intersecting this chain with $G$, we see that $G$ is filtered by a chain of normal definable (over $K$) subgroups, where we may assume that  each successive quotient is a subgroup of $({\cal U},+)$ or $({\cal U}^{*}, \times)$.  By Lemma 2.1 we may assume that either $G$ is a (finite-dimensional) differential algebraic  subgroup of $({\cal U},+)$, or  a (finite-dimensional) differential algebraic subgroup of $({\cal U}^{*},\times)$. 

Suppose first $G \leq ({\cal U},+)$.  By Cassidy \cite{Cassidy-unipotent}, $G$ (being finite dimensional)  is defined by a homogeneous linear differential equation over $K$ (in a single indeterminate $y$ but of arbitrary finite order).  So $G(K^{diff})$ is a finite-dimensional vector space over $C_{K}$ (the latter being algebraically closed) and adjoining to $K$ a basis, yields a Picard-Vessiot extension of $K$. By our assumptions such a basis must already lie in $K$ whereby $G(K^{diff})$ is isomorphic, over $K$ to $(C_{K})^d)$ for some $d$.   Suppose now that $X$ is a $PHS$ for $G$ over $K$, so $X(K^{diff})$ is a $PHS$ for $(C_{K})^{d}$ over $K$. Let $\alpha\in X(K^{diff})$, so $L = K \langle\alpha\rangle$ is clearly a strongly normal extension of $K$. For $\sigma\in Aut_{\partial}(L/K)$ there is unique $g_{\sigma}\in (C_{K})^{d}$ such that $\sigma(\alpha) = g_{\sigma}\alpha$. As $(C_{K})^{d}$ is commutative we see that the map taking $\sigma$ to $g_{\sigma}$ yields an isomorphism between $Aut_{\partial}(L/K)$ and some algebraic subgroup of $C_{K}^{d}$, in particular with a linear algebraic group (in the constants), whereby, by Corollary 2 in Section 9, Chapter VI of \cite{Kolchin-DA-AG}, $L$ is a $PV$ extension of $K$. So by our assumptions, $\alpha\in X(K)$. We have shown that $H^{1}_{\partial}(K,G) = \{1\}$. 

Now suppose $G\leq ({\cal U}^{*},\times)$. 
We can use the logarithmic derivative map from $({\cal U}^{*},\times)$ to $({\cal U},+)$ which takes $y$ to $\partial(y)/y$ to map $G$ onto a finite-dimensional definable (over $K$) subgroup of $(U,+)$ with kernel a definable (over $C_{K}$) subgroup of ${\mathbb G}_{m}({\mathcal C})$.  We can apply the previous paragraph to the image of $G$. The kernel is handled in a similar fashion (any $PHS$ for it over $K$ will yield a strongly normal, hence $PV$ extension of $K$, so has to trivial). So Lemma 2.1 tells that $H^{1}_{\partial}(K,G)$ is again trivial.  

We have completed Case (1)(a). 

\vspace{2mm}
\noindent
{\bf Case (1)(b).}
 $G = H$, namely $G$ is ``semisimple".   So $G$ is a linear (connected) $\partial$-semisimple group over $K$ in the sense of \cite{Cassidy-semisimple}.  Putting together Theorems 14 and 15 of \cite{Cassidy-semisimple}, $G$ is isogenous over $K$ to a direct product of simple (up to finite centres) noncommutative finite-dimensional linear differential algebraic groups $G_{1},..,G_{k}$, each over $K$. By Lemma 2.1, we may assume that $G$ is equal to some $G_{i}$ and is moreover centreless. By Theorem 1.5 of \cite{Pillay-DAG} $G(K^{diff})$ is isomorphic over $K^{diff}$ to a group of the form $H(C_{K^{diff}}) = H(C_{K})$ where $H$ is a simple linear algebraic group over $C_{K}$. By Proposition 4.1 of \cite{DGTIV} and our assumptions, such an isomorphism can be found defined over $K$. So we may assume $G = H({\cal C})$.  Now suppose $X$ is a (left) $PHS$ for $H({\cal C})$, over $K$.  Let again $\alpha\in X(K^{diff})$ and $L = K\langle\alpha\rangle$.  The map taking $\sigma \in Aut_{\partial}(L/K)$ to  $g_{\sigma}$ in $H(C_K)$ establishes an isomorphism between $Aut_{\partial}(L/K)$ and $B^{opp}$ for some algebraic subgroup $B$ of $H(C_{K})$. As $B^{opp}$ is also a linear algebraic group in $C_{K}$, we see again by the result of Kolchin mentioned earlier, that $L$ is a $PV$ extension of $K$. So by our assumptions $\alpha\in X(K)$.  So $X$ has a $K$-point, and we have completed Case (1)(b) and also Case (1). 

\vspace{5mm}
\noindent
{\bf Case (2).}  $G$ is arbitrary. 
\newline
We will again use model theoretic notions although purely differential algebraic notions  would also suffice. In general $G$ has $U$-rank $\omega . m + k$ for some nonnegative integers, $m, k$ and   $G$ is ``finite dimensional" iff $m= 0$. We call $G$  {\em $1$-connected} if  $G$ has no definable quotient $G/H$ (a homogeneous space) with $U(G/H)$ finite (or $G/H$ finite-dimensional). By the descending chain condition on definable subgroups (given by $\omega$-stability)  $G$ has a unique $1$-connected component, namely a smallest subgroup $G^{1,0}$ such that $G/G^{1,0}$  
is finite-dimensional. Moreover if $G$ is defined over $K$ so is $G^{1,0}$. By Corollary 6.3 of \cite{Poizat-stablegroups} ( $U$-rank inequalities for groups) and Theorem 6.7 of \cite{Poizat-stablegroups} (the Berline-Lascar decomposition),  $G$ is $1$-connected iff  it is connected and has $U$-rank of the form $\omega . m$. 
In any case by  Case (1) and Lemma 2.1, 
\newline
{\bf We may assume that $G$ is $1$-connected.}

We let $N$ be the  maximal normal solvable $1$-connected definable subgroup of $G$, clearly defined over $K$, which exists and is unique for the same reasons as in Case (1).  Then $H = G/N$ is also $1$-connected, and is ``almost semisimple"  in the sense that  $H$ has no proper nontrivial normal $1$-connected abelian definable subgroup.   By Lemma 2.1 it is enough to treat the cases where $G = N$ and where $G = H$. So: 
\newline
{\bf Case (2)(a).}  $G = N$. Exactly as in Case (1)(a), by looking at the Zariski closure $\bar G$ of $G$, a connected solvable linear algebraic group, we find a sequence $G_{i}$ of normal $K$-definable subgroups of $G$ with the  quotients  (without loss of generality) subgroups of $({\cal U},+)$ or $({\cal U},\times)$.  By Case (1) we may assume that each quotient $G_{i}/G_{i-1}$ has infinite Morley rank (or $U$-rank)  and thus has Morley rank (or $U$-rank) $\omega$. But both $({\cal U},+)$ and $({\cal U}^{*},\times)$  are connected with $U$-rank $\omega$ whereby $G_{i}/G_{i-1}$ is equal to $({\cal U},+)$ or $({\cal U}^{*}, \times)$. In particular each quotient is an algebraic group over $K$. We can now use Fact 1.6 to deduce that   $H^{1}_{\partial}(G)$ is trivial. 

\vspace{2mm}
\noindent
{\bf Case (2)(b).}  $G = H$.   Let $Z$ be the centre of $G$. By ``almost semisimplicity" of $G$, $Z$ is finite-dimensional. On the other hand using the  $1$-connectedness of $G$, $G/Z(G)$ is simple (and centreless).  So by  the inductive argument  and Case (1), we may assume that $G$ is centreless (and still linear). We may assume that the Zariski closure of $G$ (in the ambient general linear group) is semisimple and even centreless (for example by Theorem 14 of \cite{Cassidy-semisimple}).  By Theorem 15 of \cite{Cassidy-semisimple} we may assume that $G$ is simple and and its Zariski closure $\bar G$ is a simple algebraic group over $K$.  As $G$ has infinite Morley rank, we can conclude from Theorem 17 of \cite{Cassidy-semisimple}, or Proposition 5.1 of \cite{Pillay-DAG}, that $G = \bar G$, namely $G$ is an algebraic group.  Again by Fact 1.6,  $H^{1}_{\partial}(K,G)$ is trivial.  The proof is complete.

\vspace{5mm}
\noindent
{\bf Proof of Theorem 1.3.}  The nontrivial (left to right) direction of Theorem 1.3 is actually implicitly claimed in \cite{DGTII}:  Corollary 3.6 in that paper proves that if the first order theory of $K$ is superstable, $G$ is a finite Morley rank group definable over $K$, and $X$ is a $K$-definable $PHS$ for $G$, then $X(K) = X(K^{diff})$. And the sentence following the proof of that result says that it is not hard to prove (assuming $Th(K)$ superstable) that $X(K)\neq\emptyset$  whenever $X$ is a $PHS$ over $K$ for {\em any}{ definable group $G$ over $K$, although no actual proof is given.  As is readily seen the only use of superstability in the proof of Corollary 3.6 is Theorem 1.2 in that paper which says that superstability of $Th(K)$ implies that $K$ is closed under generalized strongly normal extensions.  As the proof of 3.6 in \cite{DGTII} makes heavy use of the full trichotomy theorem for $U$-rank $1$-types in $DCF_{0}$, we take the opportunity to sketch  how to extend the rather direct proof of Theorem 1.1 above to a proof of Theorem 1.3. 

First the right to left direction: If $K$ is not algebraically closed then a Galois extension of $K$ gives rise to a $PHS$ over $K$ for a finite group, as before. So we assume $K$ to be algebraically closed.  By \cite{DGTIV}, described in the last paragraph of the introduction to the current paper,  a generalized strongly normal extension $L$ of $K$ is given by adjoining a solution $\alpha\in G(K^{diff})$ of a logarithmic differential equation $dlog_{G,s}(-) = a$ on an algebraic $D$-group $(G,s)$ over $K$ with $a\in LG(K)$, where $(G,s)$ is $K$-large, namely $(G,s)^{\partial}(K^{diff}) = (G,s)^{\partial}(K)$.  As in the proof of Theorem 1.1 the map taking $\sigma\in Aut_{\partial}(L/K)$ to $\sigma(\alpha)\alpha^{-1}$  (multiplication etc. in $G$) established an isomorphism with a differential algebraic subgroup $H^{+}$ of $G(K^{diff})$, defined over $K$.  Moreover the orbit of $\alpha$ under left multiplication by $H^{+}$ is a left $PHS$ for $H^{+}$ with no $K$-point, contradicting the assumptions.

For the left to right direction we first use Fact 1.8 to obtain 
an exact sequence $1\to N \to G \to H \to 1$ of differential algebraic groups over $K$, where $N$ is linear (namely embeds over $K$ into some $GL_{n}$)  and $H$ is ``Abelian",  namely embeds over $K$ in an abelian variety. 
Hence the proof proceeds exactly as in the proof of Theorem 1.1 (left to right), using  Lemma 2.1, but with an additional case, namely when $G$ is a definable (over $K$) subgroup of a simple abelian variety $A$ (over $K$).  We will now refer to Fact 1.7(ii) of \cite{DGTIII} where the following is stated, with explanations and references. The so-called Manin map $\mu$ is a definable (over $K$) homomorphism from $A$ onto $(U,+)^{d}$ where $d = dim(A)$ and where $ker(\mu)$ is the unique smallest connected infinite definable subgroup of $A$, which we usually write as $A^{\sharp}$.  In the case when $A$ descends to the constants (so to $C_{K}$) then $\mu$ coincides with Kolchin's logarithmic derivative, and $A^{\sharp}$ is simply $A({\cal C})$, and as $C_{K}$ is algebraically closed we see that $A^{\sharp}(K) = A^{\sharp}(K^{diff})$. In the case when $A$ does not descend to the constants, then $A^{\sharp}$ is {\em strongly minimal}. Strong minimality plus the fact that $A^{\sharp}$ contains the torsion subgroup of $A$  implies that $A^{\sharp}(K) = A^{\sharp}(K^{diff})$ (see Lemma 2.2 of \cite{DGTIII}). 

Going back to our $K$-definable subgroup $G$ of $A$, let $B = \mu(A)$, and $C = G\cap A^{\sharp}$. So $B$ is a linear $DAG$ over $K$ and by Theorem 1.1, $H^{1}_{\partial}(K,B)$ is trivial. So using Lemma 2.1, it remains to show that $H^{1}_{\partial}(K,C)$ is trivial.  If $C$ is finite, we are fine as $K$ is algebraically closed. Otherwise by strong minimality of $A^{\sharp}$, $C = A^{\sharp}$.  Let $X$ be a left $PHS$ for $C$ defined over $K$. Let $\alpha\in X(K^{diff})$. Then clearly $L = K\langle\alpha\rangle$ is a generalized strongly normal extension of $K$ (with respect to $C$), as $C(K) = C(K^{diff})$), so by our assumption $\alpha\in X(K)$, and we finish.

\vspace{5mm}
\noindent
{\bf Proof of Corollary 1.4.}
\newline
We are in the context of fields $K$ equipped  with two commuting derivations $\partial_{x}$ and $\partial_{t}$.  The model-theoretic environment is $DCF_{0,2}$, mentioned earlier.  Our notation for fields of constants is now a bit different. Given a differential field $K$, $K^{\partial_{x}}$ denotes the field of constants of $K$ with respect to $\partial_{x}$, which note is a $\partial_{t}$ field. If $K$ is a model of $DCF_{0,2}$ then $(K^{\partial_{x}},\partial_{t})$ is a model of $DCF_{0}$. 

\vspace{2mm}
\noindent
As $K^{\partial_{x}}$ is algebraically closed, by \cite{Wibmer} there exists a $PPV$ extension $L$ of $K$ for the equation $\partial_{x}Y = AY$.  Let $Z$ be a fundamental matrix of solutions which generates $L$ over $K$ (as a $\{\partial_{x}, \partial_{t}\}$-field).  Let $K^{diff}$ be a differential closure of $K$ (with respect to both derivations).   Let $L_{1}$ be the differential field generated by $L$ and $(K^{diff})^{\partial_{x}}$, and let $K_{1}$ be the differential field generated by $K$ and $(K^{diff})^{\partial_{x}}$, and let $Aut(L_{1}/K_{1})$ be group of automorphisms of $L_{1}$ over $K_{1}$ (as a $\{\partial_{x}, \partial_{t}\}$-field).  If $\sigma\in Aut(L_{1}/K_{1})$ then $\sigma(Z) = Zc_{\sigma}$ for some matrix 
$c_{\sigma}\in GL_{n}((K^{diff})^{\partial_{x}})$. The map $\sigma \to c_{\sigma}$ gives an isomorphism between $Aut(L_{1}/K_{1})$ and some linear differential algebraic (with respect to $\partial_{t}$) subgroup $G$ of $GL_{n}((K^{diff})^{\partial_{x}})$ defined over $K^{\partial_{x}}$.  Now Propositions 4.25, 4.28, and Theorem 5.5 of \cite{GGO}  (which use a differential Tannakian formalism) yield a $1-1$ map from the set of $PPV$ extensions of $K$ for the equation $\partial_{x}Y = AY$, up to isomorphism over $K$ (as $\{\partial_{x}, \partial_{t}\}$-fields) to $H^{1}_{\partial}(K^{\partial x}, G)$  (with respect to the ambient differentially closed field $(K^{diff})^{\partial_{x}}, \partial_{t})$ which actually coincides with the differential closure of $(K^{\partial_{x}}, \partial_{t})$).   The assumption that $(K^{\partial_{x}}, \partial_{t})$  is both algebraically closed and $PV$-closed, together with Theorem 1.1, implies that $H^{1}_{\partial}(K^{\partial_{x}}, G)$ is trivial, whereby we obtain uniqueness of the $PPV$ extension $L$ of $K$.

\section{Additional questions and remarks.}
There are several natural environments to which one would like to extend Theorem 1.1. The first is the positive characteristic case, where the field is equipped with an iterative Hasse derivation. The model theoretic context is  the  theory of separably closed fields of characteristic $p>0$ with Ershov invariant $1$.  Even an appropriate  formulation of  the problem would be interesting as from the model theoretic point of view we are working with type-definable linear groups. 

The second is difference Galois theory, with respect to a single automorphism.  From the model-theoretic point of view, the first order theory would be $ACFA$.  An appropriate formulation of the problem would be interesting.

The third is that of several commuting derivations.  The right hand side condition would be that the constrained cohomology of any linear differential algebraic group over $K$ is trivial. The left hand side condition should be that $K$ is both algebraicaly closed and ``Picard-Vessiot closed" with respect to any set of commuting defnable derivations (and corresponding systems of linear differential equations).

We expect that in the second and third cases, the existence of ``exotic" linear difference (differential) algebraic groups implies that natural analogues of Theorem 1.1 fail.  A paper on the topic, joint with Zo\'{e} Chatzidakis, is in preparation. 

Finally: We can view Theorem 1.1 as a strong differential  analogue of the trivial fact that a field $K$ is algebraically closed iff $H^{1}(K,G)$ is trivial for every linear algebraic group over $K$.  What then would
be the differential analogue of Serre's theorem \cite{Serre} that a field $K$ has finitely many extensions of degree $n$ for all $n$ iff $H^{1}(K,G)$ is finite for every linear algebraic group $G$ over $K$?

\vspace{5mm}
\noindent
{\em Acknowledgements.}  This paper is a write-up of talks given by the author at the Differential Algebra workshop in CUNY, April 2016, and the Differential Algebra and Related Topics meeting in CUNY, October 2016.  Thanks to the organizers for the hospitality and stimulating environment.

\end{document}